\documentclass[11pt]{amsart}
\usepackage{amsmath, amsthm, amssymb,xspace}
\usepackage[breaklinks=true]{hyperref}
\usepackage{amsmath,amscd}
\usepackage{enumerate}
\usepackage{xcolor}

\makeatletter
\renewcommand\normalsize{%
\@setfontsize\normalsize\@xpt\@xiipt
\abovedisplayskip 6\p@ \@plus4\p@ \@minus6\p@
\abovedisplayshortskip \z@ \@plus4\p@
\belowdisplayshortskip 6\p@ \@plus4\p@ \@minus4\p@
\belowdisplayskip \abovedisplayskip
\let\@listi\@listI}
\makeatother
\newtheorem{theorem}{Theorem}
\newtheorem{lemma}[theorem]{Lemma}

\newtheorem*{proA}{Problem A}
\newtheorem*{proB}{Problem B}
\newtheorem*{A}{ Theorem A}
\newtheorem*{B}{ Theorem B}
\theoremstyle{definition}

\newtheorem{remark}[theorem]{Remark}

\newtheorem*{conjecture}{Conjecture}

\begin{document}

\title{Semi-Commuting Toeplitz operators on Fock-Sobolev spaces}

\author{Jie Qin$^{1,2}$}
\address{$^1$School of Mathematics and Statistics, Chongqing Technology and Business University, 400067, China}
\address{$^2$School of Mathematics and Information Science, Guangzhou University, Guangzhou, 510006, China}
\email{qinjie24520@163.com }

\thanks{The author was supported by the National Natural
Science Foundation of China (11971125, 12071155).}
\begin{abstract}
Let $F^{2,m}(\mathbb{C})$ denote the Fock-Sobolev space of complex plane. In this paper, we characterize the semi-commutator  of two Toeplitz operators on $F^{2,m}(\mathbb{C})$ is zero. The result is different from the result of Bauer, Choe, and Koo. (J. Funct. Anal., 268 (2015): 3017-3060.)
\end{abstract}
%

\maketitle
\section{Introduction}
Let $\mathbb{C}$ denote the complex plane and  $dA$ be the area measure. The Fock space $F^2$ is the space of all entire functions $f$ on $\mathbb{C}$ such that
$$||f||_2^2=\frac{1}{\pi} \int_{\mathbb{C}}|f(z)|^2e^{-|z|^2}dA(z)<\infty.$$

For any fixed non-negative integer $m$, the Fock-Sobolev space $F^{2,m}(\mathbb{C})$ consists of all entire functions $f$ on $\mathbb{C}$ such that $$\| f^{(m)}\|_2^2<\infty.$$
Cho and Zhu\cite{Cho} have shown that  $f\in F^{2,m}(\mathbb{C}^n)$ if and only if $z^\alpha f(z)$ is in $F^2(\mathbb{C}^n)$ for all multi-indices $\alpha$ with $|\alpha|=m.$ The result can of course be take as an alternative definition of $F^{2,m}.$ Let $L^2_m$ be the space of Lebesgue measurable functions $f$ on $\mathbb{C}$ such that the function $|z|^mf(z)$ is in $L^2(\mathbb{C},e^{-|z|^2}dA).$ It is well known that $L^2_m$ is a Hilbert space with the inner product $$<f,g>_m=\frac{1}{\pi m!}\int_{\mathbb{C}} f(z)\overline{g(z)}|z|^{2m}e^{-|z|^2}dA(z),\quad f,g\in L^2_m.$$
It is clear that the Fock-Sobolev space $F^{2,m}(\mathbb{C})$ is a closed subspace of the Hilbert space $L^2_m$. There is an orthogonal projection $P$ from $L^2_m$ onto $F^{2,m}(\mathbb{C})$, which is given by
$$Pf(z)=\frac{1}{\pi m!}\int_{\mathbb{C}} f(w){K_m(z,w)}|w|^{2m}e^{-|w|^2}dA(w),\quad f\in F^{2,m},$$
where $$K_m(z,w)=\sum_{k=0}^\infty \frac{m!}{(k+m)!} (z\overline{w})^k$$ is the reproducing kernel of Fock-Sobolev space $F^{2,m}(\mathbb{C}).$  See \cite{Cho} for more details about the Fock-Sobolev space.

Let $D$ denote the set of all finite linear combinations of kernel functions in $F^{2,m}(\mathbb{C}).$ It is well-known that  $D$ is dense in $F^{2,m}(\mathbb{C}).$ Let $\varphi$ be a Lebesgue measurable function  on $\mathbb{C}$  that satisfies
\begin{equation}\label{welldown}
\int_{\mathbb{C}}|\varphi(w)||K_m(z,w)||w|^{2m}e^{-|w|^2}dA(w)<\infty.
\end{equation}
So we can densely define the Toeplitz operator with the symbol $\varphi$ on $F^{2,m}(\mathbb{C})$ as follows:
\begin{align*}
T_\varphi f(z)
&=\frac{1}{\pi m!}\int_{\mathbb{C}} \varphi(w)f(w){K_m(z,w)}|w|^{2m}e^{-|w|^2}dA(w).
\end{align*}
The Hankel operator $H_\varphi$ with the symbol $\varphi$ is given by $H_\varphi f=(I-P)(\varphi f),$
where $I$ is the identity operator on $L^2_m.$ Let $k_m(z,w)$ be the normalization of the reproducing kernel $K_m(z,w),$ the Berezin transform of  $T_\varphi$ is $$\widetilde{T}_\varphi(z)=<\varphi k_m(w,z),k_m(w,z)>_m=\widetilde{\varphi}(z).$$

Let $H$ be Hibert space of analytic function. Sarason's product problem for $H$ is the following:
\textit{ Suppose $f,g\in H,$ is the boundedness of $T_fT_{\overline{g}}$ equivalent to  the boundedness of $\left[\widetilde{|f|^{2}}(z) \widetilde{|g|^{2}}(z)\right]?$}

 Sarason \cite{Sarason} originally asked this for $H=H^2$, the Hardy space of unit circle. We now call it the Sarason's
 conjecture. It was partially solved for Topelitz operators on the Hardy space in \cite{Zheng}, on the  Bergman space in \cite{Stroethoff}. Unfortunately, Sarason's conjecture is not true, both for Hardy space and Bergman space, see \cite{Aleman,Nazarov} for counterexamples. However, the Sarason's conjecture is true for  classical Fock space. A recent very impressive theorem, proved by Cho and Zhu, characterizes the Sarason's product problem for the classical Fock space, see \cite{Cho1}.

 The corresponding question for analogous Hankel operators defined on the $F^2$ was resolved by Ma-Yan-Zheng-Zhu in their paper \cite{MA}.
 They  obtained that there are functions $f(z)=e^{2\pi \texttt{i} z}$ and $g(z)=e^{z}$ such that $H_{\overline{f}}^{*} H_{\overline{g}}=0$ on $F^2$.  By definition,
 \begin{equation}\label{l111}
 H_{\overline{f}}^*H_{\overline{g}}=T_{f\overline{g}}-T_{f}T_{\overline{g}}=(T_f,T_{\overline{g}}].
 \end{equation}

Up to our knowledge,  the boundedness of a single
Toeplitz operator on $F^2$ is still an open problem, see \cite{Bauer,Berger,Berger1}. So, it is difficult to characterize the commuting Toeplitz operators on Fock spaces. In fact, the commutativity of Toeplitz operators on  $F^2$ is also an open problem. There is only one  result for the semi-commuting Toeplitz operators on Fock space $F^2$, see \cite{Bauer1}.
\begin{theorem}\label{A}
Let $\epsilon(\mathbb{C})$ be the set of entire functions $g$ on $\mathbb{C}$ such that
$|g(z)|e^{-c|z|}$ is essentially bounded on $\mathbb{C}$ for some $c>0.$ Suppose $f,g\in \epsilon(\mathbb{C}),$ then the following statements are equivalent:\\
(a) $T_fT_{\overline{g}}=T_{f\overline{g}}$.\\
(b)  $\widetilde{f\overline{g}}=f\overline{g}.$\\
(c)  Either (1) or (2) holds;

(1) $f$ or \ $g$ is constant.

(2)  There are finite collections  $\{a_j\}_{j=1}^N$  and  $\{b_l\}_{l=1}^M$ of distinct complex numbers such that
$$f \in Span \{K_{a_1},\cdots, K_{a_N}\},\quad \text{and} \quad g\in Span \{K_{b_1},\cdots, K_{b_M}\}$$
with \ $a_j\overline{b_l}= 2\eta\pi \mathrm{i}$  for each  $j$ and $l$. Here, $\eta$ is any integer and $K_z$ is the kernel of $F^2.$
\end{theorem}
In fact, the $\epsilon(\mathbb{C})$ can be characterized by the following set
$$
\mathcal{A}_1=\left\{\sum_{j=1}^{N} p_{j} K_{a_{j}}: N \in \mathbb{N} \text { and } p_{j} \in \mathcal{P}, a_{j} \in \mathbb{C} \text { for } j=1, \ldots, N\right\}.
$$
where $\mathcal{P}$ is  the space of all holomorphic polynomials on $\mathbb{C}$.
For the Fock space $F^2$, there is a unitary operator $U_w$ on $F^2$ such that $U_{w} f(z)=f(z-w)k_w(z),$ where $k_w$ is the normalized reproducing kernel of $F^2$. Using the unitary operator, it is easy to see that
\begin{align*}
\widetilde{f}(z)=\langle fk_z,k_z\rangle_{F^2}
&=\frac{1}{\pi}\int_{\mathbb{C}} f(z \pm w) e^{-|w|^2}dA (w).
\end{align*} More explicitly,  for $f,g\in \mathcal{A}_{1},$ the berezin of $f\overline{g}$ must be of the form
$$
\sum_j\widetilde{\left\{p_j \bar{\eta_j} K_{a_j} \overline{K_{b_j}}\right\}}(z)=\sum_je^{b_j \cdot \overline{a_j}} K_{a_j}(z) K_z(b_j) \eta_j^{*}\left(\partial_{z}+\overline{z}+\overline{a_j}\right) p(z+b_j),
$$
see Lemma 3.3 in \cite{Bauer1}.
The results of \cite{Bauer1} are base on the property of Berezin transform and the decomposition.

Unfortunately, we know nothing about the semi-commutant of two Toeplitz operators on Fock-Sobolev space $F^{2,m}(\mathbb{C})(m\neq0)$ from 2014 until now.
There are two natural problems, which arise from Theorem \ref{A} .
\begin{proA} Let $f,g\in D,$ what is the relationship between their symbols when two Toeplitz operators $T_f$ and $T_{\overline{g}}$
commute on $F^{2,m}(\mathbb{C})?$
\end{proA}
\begin{proB}
Is Theorem \ref{A} hold?
\end{proB}

In this paper, we will focus on the two problems. First, we will show that there are actually no nontrivial functions $f$ and $g$ in $D$ such that $H_{\overline{f}}^* H_{\overline{g}}=(T_f,T_{\overline{g}}]=0$ on $F^{2,m}(\mathbb{C})(m>0)$.
Precisely, the first main result is stated as follows.
\begin{A} Let $m$ be a positive integer, and  suppose that $f$ and $g$ are functions in $D$. Then the following statements are equivalent. \\
(a) $H_{\overline{f}}^* H_{\overline{g}}=0$ on $F^{2,m}(\mathbb{C})$ . \\
 (b) $T_{f}T_{\overline{g}}=T_{f\overline{g}}$ on $F^{2,m}(\mathbb{C})$.\\
 (c)  $\widetilde{f\overline{g}}=f\overline{g}.$\\
 (d) At least one of  $f$ and $g$ is a constant.
 \end{A}
 Every function $f \in D$ is related to the reproducing kernel by the  formula \begin{equation}\label{1}
f(z)=\sum_{k=1}^{N} c_{k} K_m\left(z, \omega_{k}\right),
\end{equation}
where $N$ is a finite positive integer and $c_k,\omega_{k}\in \mathbb{C}.$
The main obstacle of the proof of the Theorem  is to calculate the product of the Hankel operators, although every function in $D$ has the form (\ref{1}).  More specifically, the idea of the proof of the Theorem is to solve the following equations of $a_k$, $A_k$, $b_j$ and $B_j$:
$$(\sum_{k=1}^{N_1} a_{k} H_ {\overline{K_m\left(z, A_{k}\right)}}^*)(\sum_{j=1}^{N_2} b_{j} H_ {\overline{K_m\left(z, B_{j}\right)}})z^l=0,$$
where $l$ is any negative integer, and $N_1,N_2$ are two fixed  finite positive integers. This proof provide a  simpler way to prove the main Theorem in \cite{Bauer1}(Theorem \ref{A}, in this paper),  see step 1 of  Theorem \ref{zerop}.

As is well known, Fock spaces and Fock Sobolev spaces  have the same properties, since the Fock-sobolev space is the Fock space with the radial weight. For instance, Sarason's problem,  the boundedness and compactness of composition operators, the commuting Toeplitz operators with radial symbols and so on, see \cite{Bauer2,Carswell,Chen,Cho2,Cho1,Choe}. But the answer of the second question is negative.
\begin{B}
Let $m$ be a positive integer. Then  Theorem 1 doesn't hold on $F^{2,m}(\mathbb{C})$. In fact, $T_{e^{a z}} T_{\overline{e^{b z}}}=T_{e^{a z} \overline{e^{b z}}}$ if and only if $ab=0.$
\end{B}

The above results say that there is a  difference between the Fock-Sobolev space and the Fock space. The operator $U_{w}$ is a key to consider the operators on $F^2$.
However, the translations do not have such good property on Fock-Sobolev space $F^{2,m}(\mathbb{C})$, since the kernel of Fock-Sobolev is very complicated. In other words, the Berezin transform of the function can almost never be computed explicitly. Thus, the approach does not work in Fock-Sobolev space.  So, we'll look at the  Berezin transform from a different angle.

\section{Semi-Commuting Toeplitz Operators}\label{s4}
In order to prove our result, we need several lemmas.
A direct calculation gives the following lemma which we shall use often.
\begin{lemma}\label{Hz}
Let $k,j$ and $l$ be non-negative integers. Then
\begin{equation*}
T_{\overline{z}^k} z^l= \left\{
  \begin{array}{ll}
    0\quad & \hbox{if $l<k$,} \\
    \frac{(l+m)!}{(l-k+m)! } z^{l-k}\quad & \hbox{if $l\geq k;$}
  \end{array}
\right.\end{equation*}
\begin{equation*}
T_{z^l}T_{\overline{z}^k}z^j=\left\{
  \begin{array}{ll}
    0\quad & \hbox{if $j<k$,} \\
   \frac{ (j+m)!}{(j-k+m)! } z^{j+l-k}\quad & \hbox{if $j\geq k;$}
  \end{array}
\right.\end{equation*}
\begin{equation*}
T_{z^l \overline{z}^k} z^j=\left\{
  \begin{array}{ll}
   0\quad & \hbox{if $j+l<k$,} \\
    \frac{ (j+l+m)!}{(j+l-k+m)! } z^{j+l-k}\quad & \hbox{if $j+l\geq k$.}
  \end{array}
\right.\end{equation*}
 \end{lemma}
 \begin{proof}
Here we only give the proof of the frist equation. By a straightforward calculus argument,
 \begin{align*}
 T_{\overline{z}^k} z^l&=\frac{1}{\pi m!} \int_{\mathbb{C}} \overline{\omega}^k \omega^l K_m(z,\omega)|\omega|^{2m}e^{-|\omega|^2}dA(\omega)\\
 &=\frac{1}{\pi m!} \sum_{\eta=0}\frac{z^\eta}{\|z^\eta\|^2_m}  \int_{\mathbb{C}} \overline{\omega}^{k+\eta} \omega^l |\omega|^{2m} e^{-|\omega|^2}dA(\omega)\\
 &= \frac{1}{\pi m!} \sum_{\eta=0}\frac{z^\eta}{\|z^\eta\|^2_m} \int_{0}^\infty r^{1+\eta+k+l+2m}e^{-r^2}dr
 \int_{0}^{2\pi} e^{i(l-k+\eta)}d\theta\\
 &= \frac{(l+m)!}{ (l-k+m)!}z^{l-k}.
 \end{align*}
This completes the proof.
 \end{proof}
\begin{lemma}\label{lemma1}
Let $\{w_i\}_{i=1}^N$ be a finite collection of non-zero complex numbers, then $\{1-K_{m}(\cdot,w_i)\}_{i=1}^N$ is linearly independent.
\end{lemma}
\begin{proof}
Without loss of generality, assume $\{w_i\}_{i=1}^N$ is a finite collection of distinct non-zero complex numbers. Let $\sum_{i=1}^N a_i(1- K_{m}(\cdot,w_i))=0 ,$ then
$$\langle z^l , \sum_{i=1}^N a_i (1- K_{m}(z,w_i))\rangle_m=\sum_{i=1}^N a_i w_i^l=0$$ for all $l\geq1.$
Let $$M=\left(
          \begin{array}{ccc}
            w_1 & \cdots & w_N \\
            \vdots & \vdots & \vdots \\
              w_1^N  & \ldots &  w_N^N  \\
          \end{array}
        \right),\quad X=\left(
                          \begin{array}{c}
                            a_1 \\
                            \vdots \\
                            a_N \\
                          \end{array}
                        \right).
$$
So, $MX=0.$ By hypothesis, $|M|=\prod_{i=1}^N w_i \prod_{1\leq l<j\leq N}(w_j-w_l)\neq0.$ Thus, $a_i=0$ for all $1\leq i\leq N.$
\end{proof}

 Let $k,l$ and $j$ be three positive integers. Suppose $l\geq j\geq2$ and $k\geq j-1.$ For any fixed $j$, define the expression
 \begin{equation}\label{E0}
 \Xi(k)=\prod_{i=1}^j (l+k+m-j+i)=C_0+\sum_{i=1}^{j}C_i\bigg\{ \prod_{\lambda=0}^{i-1} (k+m-\lambda)\bigg\},
 \end{equation}
 where $C_i$  is a positive number independent of $k$, but dependent  on  $l,$ $m$ and $j$. Note that $C_{j-1}=jl$ and $C_j=1.$
 It is clear that
 \begin{align}\label{E4}
 \Xi(j)=\prod_{i=1}^j (l+m+i)=\sum_{i=0}^{j}C_i \frac{(m+j)!}{(m+j-i)!}.
 \end{align}

 There are non-constant functions $f$ and $g$ in the set of all finite linear combinations of kernel functions of  $F^2({\mathbb{C}})$ such that $H_{\overline{f}}^*H_{\overline{g}}=0$ on  Fock space $F^2({\mathbb{C}})$, see \cite{MA}.  But it's not true in Fock-Sobolev space, see Theorem \ref{zerop}.  This means that there is a huge difference between Fock spaces and Fock-Sobolev spaces. Before we prove this result, we establish the following special case first.
 \begin{lemma}\label{ex}
 Let $j$ and $l$ be two non-negative integers such that $l\geq j\geq2$. Let $A$ and  $B$ be two non-zero constants. For any fixed $j$, define the expression
 \begin{align}\label{E-1}
C(j,l)=&\sum_{k\geq0}\frac{m!}{(k+m)!}  \frac{(l+k+m)!}{(k+l-j+m)!}A^k B^{l+k-j}\nonumber\\
&-\sum_{0\leq k\leq j}  \frac{m!}{(k+m)!}  \frac{(j+m)!}{(k+l-j+m)!} \frac{(l+m)!}{(j-k+m)!}A^k B^{l+k-j}.
\end{align}
Suppose \begin{align*}
m!\frac{e^{AB}-q_m(AB)}{(AB)^m}=\sum_{k\geq 0}\frac{m!(AB)^k}{(k+m)!}=1,
\end{align*}
where $q_0=0$ and $q_m$ is the Taylor polynomial of $e^{AB}$ of order $m-1$ for all $m\geq1$.
Then $$C(j,l)=\sum_{k=1}^{j-1}Q_k A^k B^{l+k-j}-\sum_{i=0}^{j-2}C_i A^iB^{l+i-j}\bigg\{\sum_{k=1}^{j-i-1} \frac{m!(AB)^k}{(k+m)!}\bigg\},$$
where the $C_i$ is defined as before (\ref{E0}). Moreover,
 $$Q_k=\frac{m!}{(k+m)!}  \frac{(l+k+m)!}{(k+l-j+m)!}-\frac{m!}{(k+m)!}  \frac{(j+m)!}{(k+l-j+m)!} \frac{(l+m)!}{(j-k+m)!} $$
 for $1\leq k\leq j-1.$
 \end{lemma}
 \begin{proof}
 For simplicity, we write
 $$J=\sum_{k=1}^{j-1}Q_k A^k B^{l+k-j},\quad E =m!\frac{e^{AB}-q_m(AB)}{(AB)^m},$$
 $$M=\sum_{i=0}^{j-2}C_i A^iB^{l+i-j}\bigg\{\sum_{k=1}^{j-i-1} \frac{m!(AB)^k}{(k+m)!}\bigg\},$$
 and
 $$I=\sum_{k\geq j+1}\frac{m!}{(k+m)!}  \frac{(l+k+m)!}{(k+l-j+m)!}A^k B^{l+k-j}.$$
In fact, for $l\geq j\geq 2,$
\begin{equation}\label{E1}
C(j,l)=I+J+(\prod_{i=1}^j \frac{l+m+i}{m+i}-1)A^j B^l.
 \end{equation}
 By (\ref{E0}) and a simple computation,
 \begin{align*}
I&=\sum_{k\geq j+1}\frac{m!}{(k+m)!}\bigg(\prod_{i=1}^j (l+k+m-j+i)\bigg)   A^k B^{l+k-j}\nonumber\\
 &=\sum_{k\geq j+1}\frac{m!}{(k+m)!}\bigg\{C_0+\sum_{i=1}^{j}C_i \prod_{\lambda=0}^{i-1} (k+m-\lambda)\bigg\} A^k B^{l+k-j}\nonumber\\
 &=\sum_{i=0}^jC_i\bigg\{ E-\sum_{k=0}^{j-i} \frac{m!(AB)^k}{(k+m)!} \bigg\}A^iB^{l+i-j}\nonumber\\
 &=\sum_{i=0}^{j-1}C_i\bigg\{ E-1-\sum_{k=1}^{j-i} \frac{m!(AB)^k}{(k+m)!} \bigg\}A^iB^{l+i-j}+C_j (E-1)A^jB^l\nonumber\\
 \end{align*}
 where the $C_i$ is defined as before (\ref{E0}). Since $E=1$, then
  \begin{align}\label{E2}
  I&=-\sum_{i=0}^{j-1}C_i A^iB^{l+i-j}\bigg\{\sum_{k=1}^{j-i} \frac{m!(AB)^k}{(k+m)!}\bigg\}\nonumber\\
 &=-\sum_{i=0}^{j-2}C_i A^iB^{l+i-j}\bigg\{\sum_{k=1}^{j-i-1} \frac{m!(AB)^k}{(k+m)!}\bigg\}-\sum_{i=0}^{j-1} \frac{m!C_i}{(j-i+m)!}A^j B^l,
   \end{align}
 By  (\ref{E1}) and (\ref{E2}),

 \begin{align*}
C(j,l)&=J-M+\bigg\{\prod_{i=1}^j \frac{l+m+i}{m+i}-1-\sum_{i=0}^{j-1} \frac{m!C_i}{(j-i+m)!}\bigg\}A^j B^l.
 \end{align*}
Since $C_j=1,$ it is immediate from (\ref{E4}) that
\begin{align*}
&\quad\prod_{i=1}^j \frac{l+m+i}{m+i}-1-\sum_{i=0}^{j-1} \frac{m!C_i}{(j-i+m)!}\\
&=\prod_{i=1}^j\frac{1}{m+i}\bigg\{\prod_{i=1}^j (l+m+i)-\prod_{i=1}^j(m+i)-\sum_{i=0}^{j-1} \frac{C_i (m+j)!}{(j-i+m)!}\bigg\}\\
&=\prod_{i=1}^j\frac{1}{m+i}\bigg\{\prod_{i=1}^j (l+m+i)-\sum_{i=0}^{j} \frac{C_i (m+j)!}{(j-i+m)!}\bigg\}\\
&=0.
 \end{align*}
Thus $C(j,l)=J-M.$
This is the desired result.
 \end{proof}
 The coefficients of $A^k B^{l+k-j}$ can almost never be computed explicitly for all $l\geq j\geq2, 1\leq k\leq j.$ Lemma \ref{ex} shows that the coefficient of  $A^jB^l$ is zero. Fixed $j\geq2,$ we now  consider the coefficient of $A^{j-1} B^{l-1}$ for all $l\geq j.$
  \begin{lemma}\label{re1}
For $l\geq j\geq2,$ we let $ \Theta$ denote the coefficient of $ A^{j-1} B^{l-1}.$
Then $\Theta\neq0$ if $m\neq0.$
 \end{lemma}
 \begin{proof}
 It follows from Lemma  \ref{ex} that $$\Theta=Q_{j-1}-\sum_{i=0}^{j-2}C_i \frac{m!}{(j-i-1+m)!},$$
    where $C_i$ as defined in (\ref{E0}).
It is obvious that  \begin{align*}
 Q_{j-1}&=\frac{m!}{(j-1+m)!}\frac{(l+j-1+m)!}{(l-1+m)!}-\frac{m!}{(j-1+m)!}\frac{(j+m)!(l+m)}{(1+m)!}\\
&=\prod_{\lambda=1}^{j-1}\frac{1}{m+\lambda}\bigg\{\prod_{i=1}^j (l+m+i-1)-\frac{(j+m)!(l+m)}{(1+m)!}\bigg\}.
\end{align*}
By (\ref{E0}), we have
\begin{align*}
\Xi(j-1)&=\prod_{i=1}^{j}(l+m+i-1)=C_0+\sum_{i=1}^{j}C_i\bigg\{ \prod_{\lambda=0}^{i-1} (j-1+m-\lambda)\bigg\}\\
&=\sum_{i=0}^{j-2}C_i \frac{(m+j-1)!}{(j-i-1+m)!}+C_{j-1}\prod_{i=1}^{j-1}(m+i)+\prod_{i=0}^{j-1}(m+i),
\end{align*}
then
 \begin{align*}
&\quad \sum_{i=0}^{j-2}C_i \frac{m!}{(j-i-1+m)!}\\
&=\prod_{\lambda=1}^{j-1} \frac{1}{m+\lambda}\bigg\{\sum_{i=0}^{j-2}C_i \frac{(m+j-1)!}{(j-i-1+m)!}\bigg\}\\
&=\prod_{\lambda=1}^{j-1} \frac{1}{m+\lambda}\bigg\{ \prod_{i=1}^{j}(l+m+i-1)\\
&\quad-C_{j-1}\prod_{i=1}^{j-1}(m+i)-\prod_{i=0}^{j-1}(m+i)\bigg\}.
  \end{align*}
By above equations, then
  \begin{align*}\Theta&=\prod_{\lambda=1}^{j-1}\frac{1}{m+\lambda}\bigg\{\prod_{i=1}^{j}(l+m+i-1)-\frac{(j+m)!(l+m)}{(1+m)!}
\\
&\quad-\prod_{i=1}^{j}(l+m+i-1)+C_{j-1}\prod_{i=1}^{j-1}(m+i)+
  \prod_{i=0}^{j-1}(m+i)\bigg\}\\
  &=\prod_{\lambda=1}^{j-1}\frac{1}{m+\lambda}\bigg\{-(l+m)\prod_{i=2}^j(m+i)+jl\prod_{i=1}^{j-1}(m+i)+
  \prod_{i=0}^{j-1}(m+i)\bigg\}\\
&=\prod_{\lambda=1}^{j-1}\frac{1}{m+\lambda}\prod_{i=1}^{j-1}(m+i)\bigg\{-\frac{(l+m)(j+m)}{m+1}+jl+m\bigg\}\\
&=\frac{m(j-1)(l-1)}{m+1}
  \end{align*}
  and hence we have $\Theta\neq0$ if $m\neq0.$
 \end{proof}
 Lemma \ref{re1} shows that the coefficient of $A^{j-1} B^{l-1}$ is non-zero if $m\neq0$.  Easily modifying the proof of Theorem 9 in \cite{MA}, we get $H_{\overline{f}}^* H_{\overline{g}}=0$ if and only if
 $\widetilde{fg}=f\overline{g}.$ So, we omit its proof here. Recall that  $D$ is the set of all finite linear combinations of kernel functions in $F^{2,m}(\mathbb{C}).$ We now prove Theorem \ref{A}.

 \begin{theorem}\label{zerop}
Let $f$ and $g$ be functions in $D.$ Suppose $m\neq0$, then $H_{\overline{f}}^* H_{\overline{g}}=(T_f,T_{\overline{g}}]=0$ on $F^{2,m}(\mathbb{C})$ if and only if at least one of  $f$ and $g$ is a constant function.
 \end{theorem}
\begin{proof}
It is clear that $H_{\overline{f}}^* H_{\overline{g}}=0$ if one of $f$ and $g$ is constant function.

Conversely, suppose $H_{\overline{f}}^* H_{\overline{g}}=0,$ then $T_f T_{\overline{g}}=T_{f\overline{g}}$ by (\ref{l111}).  We set
$f(z)=\sum_{i=1}^{N_1} a_i K_m(z,A_i)$ and $g(z)=\sum_{\lambda=1}^{N_2} b_\lambda K_m(z,B_\lambda),$ where $N_1$, $N_2$ are finite positive integers, and $a_i,A_i,b_\lambda,B_\lambda\in \mathbb{C}$. We can  rewrite $$f(z)=\sum_{i=1}^{N_1}\sum_{k_i=0}^\infty a_i\frac{m!}{(k_i+m)!}\overline{A}_i^{k_i}z^{k_i}$$ and
$$g(z)=\sum_{\lambda=1}^{N_2}\sum_{\ell_\lambda=0}^\infty b_\lambda\frac{m!}{(\ell_\lambda+m)!}\overline{B}_\lambda^{\ell_\lambda}z^{\ell_\lambda}.$$
By Lemma \ref{Hz}, we have
\begin{align*}
T_{f(z)} T_{\overline{g}(z)}z^l&=\sum_{i=1}^{N_1}\sum_{\lambda=1}^{N_2}\sum_{k_i=0}^\infty\sum_{0\leq\ell_\lambda\leq l} a_i \overline{b}_\lambda \frac{m!}{(k_i+m)!}\frac{m!}{(\ell_\lambda+m)!} \\
&\quad \times\frac{(l+m)!}{(l-\ell_\lambda+m)!} \overline{A}_i^{k_i}B_\lambda^{\ell_\lambda}z^{l+k_i-\ell_\lambda}
\end{align*}
and \begin{align*}
 T_{f(z)\overline{g}(z)}z^l&=\sum_{i=1}^{N_1}\sum_{\lambda=1}^{N_2}\sum_{k_i=0}^\infty \sum_{0\leq\ell_\lambda\leq l+k_i} a_i \overline{b}_\lambda \frac{m!}{(k_i+m)!}\frac{m!}{(\ell_\lambda+m)!} \\
&\quad \times\frac{(l+k_i+m)!}{(l+k_i-\ell_\lambda+m)!} \overline{A}_i^{k_i}B_\lambda^{\ell_\lambda}z^{l+k_i-\ell_\lambda}.
 \end{align*}

Let $l+k_i-\ell_\lambda=j\geq0,$ then $\ell_\lambda=l+k_i-j$ and $k_i\geq \text{max}\{0, j-l\}.$ If $\ell_\lambda\leq l$, then $k_i\leq j$.
By the above two equations, we obtain $T_f T_{\overline{g}}=T_{f\overline{g}}$ if and only if
\begin{align}\label{QQQQ}
& \quad\sum_{i=1}^{N_1}\sum_{\lambda=1}^{N_2} \sum_{max(0,j-l)\leq k_i\leq j}a_i \overline{b}_\lambda  \frac{m!}{(k_i+m)!}  \frac{(j+m)!}{(k_i+l-j+m)!} \nonumber\\
&\quad \times \frac{(l+m)!}{(j-k_i+m)!}\overline{A}_i^{k_i} B_\lambda^{l+k_i-j}\nonumber\\
&= \sum_{i=1}^{N_1}\sum_{\lambda=1}^{N_2} \sum_{max(0,j-l)\leq k_i}^\infty a_i \overline{b}_\lambda\frac{m!}{(k_i+m)!}  \frac{(l+k_i+m)!}{(k_i+l-j+m)!}\overline{A}_i^{k_i} B_\lambda^{l+k_i-j}
\end{align}
for all $j,l\geq0.$

We define
\begin{align*}
C(j,l,i,\lambda)&=\sum_{0\leq k_i}^\infty \frac{m!}{(k_i+m)!}  \frac{(l+k_i+m)!}{(k_i+l-j+m)!}\overline{A}_i^{k_i} B_\lambda^{l+k_i-j}\\
&\quad-\sum_{0\leq k_i\leq j}\frac{m!}{(k_i+m)!}  \frac{(j+m)!}{(k_i+l-j+m)!} \frac{(l+m)!}{(j-k_i+m)!}\overline{A}_i^{k_i} B_\lambda^{l+k_i-j}.
\end{align*}
Note that $C(j,l,i,\lambda)$ is independent of $i$ and $\lambda$. In other words, if we set $k=k_i$, $A=\overline{A}_i$ and $B=B_\lambda$ in (\ref{E-1}), then $C(j,l,i,\lambda)=C(j,l),$ where $C(j,l)$ as defined in (\ref{E-1}).

Without loss of generality, assume $j\leq l$.
By (\ref{QQQQ}), we obtain
\begin{equation}\label{E-2}
0=\sum_{i=1}^{N_1}\sum_{\lambda=1}^{N_2}a_i\overline{b}_\lambda C(j,l,i,\lambda).
\end{equation}

Because of its length, the proof will be divided into several steps.\\
\textbf{Step 1.} We claim that $$1=m! \frac{e^{\overline{A}_iB_\lambda}-q_m(\overline{A}_iB_\lambda)}{(\overline{A}_iB_\lambda)^m}, \quad 1\leq i \leq N_1, 1\leq\lambda \leq N_2.$$

To prove this claim, assume $a_i,A_i, b_\lambda,B_\lambda\neq0$ where $1\leq i\leq N_1$ and $1\leq j\leq N_2$. If $l\geq j=0$, then (\ref{E-2}) becomes
\begin{align}\label{QQQQQQ}
\sum_{i=1}^{N_1}\sum_{\lambda=1}^{N_2} a_i \overline{b}_\lambda  B_\lambda^{l}
&=\sum_{i=1}^{N_1}\sum_{\lambda=1}^{N_2} a_i \overline{b}_\lambda B_\lambda ^l m! \frac{e^{\overline{A}_iB_\lambda}-q_m(\overline{A}_iB_\lambda)}{(\overline{A}_iB_\lambda)^m}
\end{align}
for all $l\geq0.$ By (\ref{QQQQQQ}),
\begin{equation}\label{equation1}
\sum_{i=1}^{N_1}\sum_{\lambda=1}^{N_2}   a_i\overline{b}_\lambda (1-m!\frac{e^{\overline{A}_iB_\lambda}-q_m(\overline{A}_iB_\lambda)}{(\overline{A}_iB_\lambda)^m})B_\lambda^{l}=0, \quad l\geq0.
\end{equation}
Now assume $\{B_\lambda\}$ be a finite collection of distinct non-zero complex numbers. Let
   $$E_1=(c_{k,\lambda})_{1\leq k,\lambda\leq N_2}=\bigg( \sum_{i=1}^{N_1} a_i B_\lambda^k (1-m!\frac{e^{\overline{A}_iB_\lambda}-q_m(\overline{A}_iB_\lambda)}{(\overline{A}_iB_1\lambda)^m}\bigg)_{1\leq k,\lambda\leq N_2}$$
and $$X_1=\left(
         \begin{array}{c}
           b_1 \\
           \vdots \\
           b_{N_2} \\
         \end{array}
       \right).
$$ It follows from (\ref{equation1}) that $E_1X_1=0$. By the property of Vandermonde determinant,
\begin{align*}
0&=|E_1|=  \prod_{1\leq \iota< \kappa\leq N_2 }^{N_2}(B_\iota-B_\kappa) \prod_{\lambda=1}^{N_2}B_\lambda\bigg(\sum_{i=1}^{N_1} a_i  (1-m!\frac{e^{\overline{A}_iB_\lambda}-q_m(\overline{A}_iB_\lambda)}{(\overline{A}_iB_\lambda)^m})\bigg) \\
&=\prod_{\lambda=1}^{N_2}\bigg(\sum_{i=1}^{N_1} a_i  (1-m!\frac{e^{\overline{A}_iB_\lambda}-q_m(\overline{A}_iB_\lambda)}{(\overline{A}_iB_\lambda)^m})\bigg).
\end{align*}
 This implies there is a $\kappa$ so that $$\sum_{i=1}^{N_1} a_i  (1-m!\frac{e^{\overline{A}_iB_\kappa}-q_m(\overline{A}_iB_\kappa)}{(\overline{A}_iB_\kappa)^m}=0.$$
 By the same way, we have  $$\sum_{i=1}^{N_1} a_i  (1-m!\frac{e^{\overline{A}_iB_\lambda}-q_m(\overline{A}_iB_\lambda)}{(\overline{A}_iB_\lambda)^m})=\sum_{i=1}^{N_1} a_i (1-K_m(B_\lambda,A_i))=0 $$ for all
$1\leq \lambda \leq N_2.$

Since every $a_i$ is non-zero constant, then, by Lemma \ref{lemma1}, $$1-K_m(B_\lambda,A_i)=1-m!\frac{e^{\overline{A}_iB_\lambda}-q_m(\overline{A}_iB_\lambda)}{(\overline{A}_iB_\lambda)^m}=0$$ for all
$1\leq \lambda \leq N_2.$  In this case, $B_\lambda$ will be the
independent variable.

In fact, if we set $j\geq l=0,$ then, by (\ref{QQQQ}),
\begin{align*}\label{QQQQQ}
\sum_{i=1}^{N_1}\sum_{\lambda=1}^{N_2}   a_i\overline{b}_\lambda \overline{A}_i^{j}
&= \sum_{i=1}^{N_1}\sum_{\lambda=1}^{N_2}   a_i\overline{b}_\lambda \overline{A}_i^{j}m! \frac{e^{\overline{A}_iB_\lambda}-q_m(\overline{A}_iB_\lambda)}{(\overline{A}_iB_\lambda)^m}
\end{align*}
for all $j\geq0.$ Exchanging the roles of $A_i$ and $B_\lambda$, we see that
$$1-m!\frac{e^{\overline{A}_iB_\lambda}-q_m(\overline{A}_iB_\lambda)}{(\overline{A}_iB_\lambda)^m}=0$$ for all
$1\leq i \leq N_1.$
This shows that
\begin{equation*}
m! \frac{e^{\overline{A}_iB_\lambda}-q_m(\overline{A}_iB_\lambda)}{(\overline{A}_iB_\lambda)^m}=1, \quad 1\leq i \leq N_1, 1\leq\lambda \leq N_2.
\end{equation*}
Now we can prove this theorem by Lemma \ref{ex}.\\
\textbf{Step 2.}
To prove that \begin{align*}
\sum_{i=1}^{N_1}\sum_{\lambda=1}^{N_2}a_i\overline{b}_\lambda \overline{A}_i B_\lambda^{l-1}=0
\end{align*} for all $l\geq2$.

Let us now assume $l\geq j=2.$  Using (\ref{E-2}), Lemma \ref{ex} and \ref{re1}, we have
\begin{align*}
0&= \sum_{i=1}^{N_1}\sum_{\lambda=1}^{N_2} a_i \overline{b}_\lambda\frac{m(l-1)\overline{A}_i B_\lambda^{l-1}}{m+1}
\end{align*}
for all $l\geq2$. It follows that \begin{align}\label{ex1}
\sum_{i=1}^{N_1}\sum_{\lambda=1}^{N_2}a_i\overline{b}_\lambda \overline{A}_i B_\lambda^{l-1}=0
\end{align} for all $l\geq2$. This completes Step 2.\\
\textbf{Step 3.} We claim $$\sum_{i=1}^{N_1}a_i\overline{A}_i^{l}=0\quad \text{or}\quad \sum_{\lambda=1}^{N_2} \overline{b}_\lambda B_{\lambda}^{l}=0$$
for all $l\geq1.$

To prove this claim,  we assume $l\geq j=3$. By lemma \ref{ex} and (\ref{E-2}), $$0=\sum_{i=1}^{N_1}\sum_{\lambda=1}^{N_2}a_i\overline{b}_\lambda\bigg\{\eta_1\overline{A}_iB_{\lambda}^{l-2}+\eta_2\overline{A}_i^2B_{\lambda}^{l-1}+
\eta_3\overline{A}_i^3B_{\lambda}^{l}
\bigg\}$$ where
$\eta_k$ is the coefficient of $\overline{A}_i^kB_\lambda^{l+k-3}(1\leq k\leq3)$. Note that $\eta_k$ is independent
of $i$ and $\lambda.$
By  Lemma \ref{re1} and (\ref{ex1}), we have  $$\sum_{i=1}^{N_1}\sum_{\lambda=1}^{N_2}a_i\overline{b}_\lambda\eta_1\overline{A}_iB_{\lambda}^{l-2}=0,\quad \eta_2\neq0,\quad \eta_3=0$$
for all $l\geq3.$
Then $$\sum_{i=1}^{N_1}\sum_{\lambda=1}^{N_2}a_i\overline{b}_\lambda \overline{A}_i^2 B_\lambda^{l-1}=0,\quad l\geq3.$$
By the iteration method,  we get
\begin{align*}
\sum_{i=1}^{N_1}\sum_{\lambda=1}^{N_2}a_i\overline{b}_\lambda \overline{A}_i^{j-1}B_{\lambda}^{l-1}=0
\end{align*} for all
$l\geq j\geq2.$ In  particular, we have
\begin{align*}
\sum_{i=1}^{N_1}\sum_{\lambda=1}^{N_2}a_i\overline{b}_\lambda \overline{A}_i^{l}B_{\lambda}^{l}=0
\end{align*}
for all $l\geq1.$ This shows that $$\sum_{i=1}^{N_1}a_i\overline{A}_i^{l}=0\quad \text{or}\quad \sum_{\lambda=1}^{N_2} \overline{b}_\lambda B_{\lambda}^{l}=0$$
for all $l\geq1.$ This completes Step 3.\\
\textbf{Step 4.} We now complete the proof of the theorem. For any $l\geq1,$ we do only the case $\sum_{i=1}^{N_1}a_i\overline{A}_i^{l}=0.$ If $\sum_{i=1}^{N_1}a_i\overline{A}_i^{l}=0$ for all $l\geq1$, then we have $D_1X_1=0$, where $$D_1=\left(
                                                                \begin{array}{cccc}
                                                                  \overline{A}_1 & \cdots & \cdots & \overline{A}_{N_1} \\
                                                                  \overline{A}_1^2 &\cdots & \cdots & \overline{A}_{N_1}^2 \\
                                                                  \vdots & \vdots & \vdots & \vdots \\
                                                                  \overline{A}_1^{N_1} & \cdots & \cdots &  \overline{A}_{N_1}^{N_1}\\
                                                                \end{array}
                                                              \right), \quad X_1=(a_1,a_2\cdots,a_{N_1})^T.$$
Note that if $N_1=1,$ then $a_1\overline{A}_1=0$. This contradiction shows that at least one of $a_1$ and $A_1$ is 0, then $f$ is a constant function since $f(z)=a_1K(z,A_1).$

Suppose $N_1\geq 2.$ Since $a_i\neq0$ for $1\leq i \leq N_1,$ then $|D_1|=0.$ In fact, $|D_1|$ is a Vandermond determinant with the first row removed, then $$\prod_{i=1}^{N_1}\overline{A}_i\prod_{1\leq \iota< \kappa \leq N_1}(\overline{A}_\kappa-\overline{A}_\iota)=0.$$
Since $A_i\neq0$ for $1\leq i \leq N_1,$ then there are $\kappa_1$ and $\kappa_2$ such that $A_{\kappa_1}=A_{\kappa_2}$ where $\kappa_1\neq\kappa_2$ and $1\leq\kappa_1,\kappa_2\leq N_1$. So we can rewrite
$$f(z)=a_1 K(z,A_1)+\cdots+(a_{\kappa_1}+a_{\kappa_2})K(z,A_{\kappa_1})+\cdots+a_{N_1}K(z,A_{N_1}).$$
If $a_{\kappa_1}+a_{\kappa_2}=0,$ we will omit $K(z,A_{\kappa_1})$, so we may assume $a_{\kappa_1}+a_{\kappa_2}\neq0.$

Without loss of generality, assume  $1<\kappa_1,\kappa_2\leq N_1$  and $a_{\kappa_1}+a_{\kappa_2}\neq0.$
By hypothesis, we can obtain $D_2X_2=0$, where $$D_2=\left(
                                                                       \begin{array}{ccccc}
                                                                         \overline{A}_1 & \cdots &A_{\kappa_1} & \cdots & \overline{A}_{N_1} \\
                                                                          \overline{A}_1^2 & \cdots & A_{\kappa_1}^2 & \cdots & \overline{A}_{N_1}^2 \\
                                                                         \vdots & \cdots & \vdots & \cdots & \cdots \\
                                                                         \vdots & \cdots & \cdots & \cdots & \cdots \\
                                                                          \overline{A}_1^{N_1-1} & \cdots & \overline{A}_{\kappa_1}^{N_1-1} & \cdots &\overline{A}_{N_1}^{N_1-1} \\
                                                                       \end{array}
                                                                     \right)
                                                               $$ and $$ X_2=(a_1,a_2\cdots,(a_{\kappa_1}+a_{\kappa_2}),\cdots,a_{N_1})^T.$$
Note that if $N_1=3,$ then $$\left\{
                             \begin{array}{ll}
                               a_1 \overline{A}_1+(a_{\kappa_1}+a_{\kappa_2})\overline{A}_{\kappa_1}=0  \\
                               a_1 \overline{A}_1^2+(a_{\kappa_1}+a_{\kappa_2})\overline{A}_{\kappa_1}^2=0,
                             \end{array}
                           \right.$$
and hence we have $a_1=-(a_{\kappa_1}+a_{\kappa_2})$ and $A_1=A_{\kappa_1}=A_{\kappa_2}.$ Thus $f=0.$ We have used the fact $\{a_{\kappa_1},a_{\kappa_2}\}=\{a_2,a_3\}.$

Suppose $N_1 \geq 4.$ For the same reason, there  are  $\kappa_3$ and $\kappa_4$ such that $A_{\kappa_3}=A_{\kappa_4},$  where $\kappa_3\neq\kappa_4$ and $1\leq\kappa_3,\kappa_4\leq N_1.$ Finally, we conclude that
\begin{align}\label{ex4}
A_1=A_2=\cdots=A_{N_1}
\end{align}
if  the sum of any $k(2\leq k\leq N_1-1)$ elements in $\{a_i:1\leq i\leq N_1\}$ is not equal to 0. By assumption, $\sum_{i=1}^{N_1}a_i=0$ for $N_1\geq2,$
this together with (\ref{ex4}) show that $f=0,$ since $f(z)=\sum_{i=1}^{N_1}a_i K(z,A_i).$

On the other hand, suppose that $T_\varphi T_{\overline{\psi}}= T_{\varphi\overline{\psi}}$, then $$T_{\varphi+c_1} T_{\overline{\psi}+c_2}= T_{(\varphi+c_1)(\overline{\psi}+c_2)},$$
where $c_1$ and $c_2$ are constants.
This together with preceding proof show that $f$ or $g$ is a constant function if $T_fT_{\overline{g}} = T_{f\overline{g}}$. The proof is complete.
\end{proof}
\begin{remark}\label{RE1}
 The proof of Theorem \ref{zerop} shows that the hypothesis $ m\neq0$ is necessary. In other words, the (\ref{ex1}) doesn't hold if $m=0.$
Then the iteration method doesn't work, since we have lost the initial value.
\end{remark}
Now, we now proceed to answer the second question.
\begin{theorem}
Let $m$ be a positive integer. Suppose $a,b\in \mathbb{C},$ then $T_{e^{a z}} T_{\overline{e^{b z}}}=T_{e^{a z} \overline{e^{b z}}}$ on $F^{2,m}(\mathbb{C})$ if and only if $ab=0.$
\end{theorem}
\begin{proof}
If $ab=0,$ it is clear that $T_{e^{a z}} T_{\overline{e^{b z}}}=T_{e^{a z} \overline{e^{b z}}}$. Now assume $ab\neq0.$ then, by (\ref{QQQQ}),
\begin{align*}
&  \sum_{k=\max \{0, j-l\}} \frac{(k+l+m) !}{k !(k+l-j) !} a^{k} \overline{b}^{k+l-j} \\
=&  \sum_{k=\max \{0, j-l\}}^{j}  \frac{(j+m) !}{k !(k+l-j) !} \frac{(l+m) !}{(j-k_i+m) !} a^{k} \overline{b}^{k+l-j}
\end{align*}
Let $0=l\leq j,$ we obtain \begin{align*}
0 =\sum_{k=1}^{\infty} \frac{a^{k} \bar{b}^{k}}{k !} \frac{(k+j+m) !}{(k+j) !}.
\end{align*}
Using the equation, we now define $$\rho(j)=\sum_{k=1}^{\infty} \frac{a^{k} \bar{b}^{k}}{k !}\prod_{i=1}^m(k+j+i)=0.$$
A simple calculation gives
\begin{equation}\label{EQ1}\left\{
                          \begin{array}{ll}
\rho(j)-\rho(j-1)=\sum_{k=1}^{\infty} \frac{ma^{k} \bar{b}^{k}}{k !}\prod_{i=2}^{m}(k+j-1+i)=0; \\
\rho(j+1)-\rho(j)=\sum_{k=1}^{\infty} \frac{ma^{k} \bar{b}^{k}}{k !}\prod_{i=2}^{m}(k+j+i)=0;\\
\rho(j+2)-\rho(j+1)=\sum_{k=1}^{\infty} \frac{ma^{k} \bar{b}^{k}}{k !}\prod_{i=2}^{m}(k+j+1+i)=0;\\
\rho(j+3)-\rho(j+2)=\sum_{k=1}^{\infty} \frac{ma^{k} \bar{b}^{k}}{k !}\prod_{i=2}^{m}(k+j+2+i)=0;\\
\vdots
                          \end{array}
                        \right.\end{equation}
By (\ref{EQ1}),
\begin{equation}\label{EQ5}\left\{
                          \begin{array}{ll}
\frac{\rho(j+1)-2\rho(j)+\rho(j-1)}{m(m-1)}=\sum_{k=1}^{\infty} \frac{a^{k} \bar{b}^{k}}{k !}\prod_{i=3}^{m}(k+j-1+i)=0;\quad\natural\\
      \frac{\rho(j+2)-2\rho(j+1)+\rho(j)}{m(m-1)}=\sum_{k=1}^{\infty} \frac{a^{k} \bar{b}^{k}}{k !}\prod_{i=3}^{m}(k+j+i)=0; \quad \sharp\\
       \frac{\rho(j+3)-2\rho(j+2)+\rho(j+1)}{m(m-1)}=\sum_{k=1}^{\infty} \frac{a^{k} \bar{b}^{k}}{k !}\prod_{i=3}^{m}(k+j+1+i)=0; \quad \flat\\
\vdots

         \end{array}
                        \right.\end{equation}
It follows from   (\ref{EQ5}) that
\begin{equation}\label{EQ6}
\left\{
                          \begin{array}{ll}
\frac{\sharp-  \natural}{m-2} =\sum_{k=1}^{\infty} \frac{a^{k} \bar{b}^{k}}{k !}\prod_{i=4}^{m}(k+j-1+i)=0;\\
\frac{\flat-\sharp}{m-2}=\sum_{k=1}^{\infty} \frac{a^{k} \bar{b}^{k}}{k !}\prod_{i=4}^{m}(k+j+i)=0;\\
\vdots
\end{array}
                        \right.
\end{equation}
     Using (\ref{EQ1}), (\ref{EQ5}) and (\ref{EQ6}), we can obtain that
\begin{equation}\label{EQ3}\left\{
                          \begin{array}{ll}
                            \sum_{k=1}^{\infty} \frac{a^{k} \bar{b}^{k}}{k !}(k+m)=0, \\
                            \sum_{k=1}^{\infty} \frac{a^{k} \bar{b}^{k}}{k !}\prod_{i=m-1}^{m}(k+i)=0.
                          \end{array}
                        \right.\end{equation}
It follows (\ref{EQ3}) that \begin{equation}\label{EQ4}\left\{
                          \begin{array}{ll}
                            a\overline{b}e^{a\overline{b}}+m(e^{a\overline{b}}-1) =0, \\
                            (a\overline{b})^2e^{a\overline{b}}+2m a\overline{b}e^{a\overline{b}}+m(m-1)(e^{a\overline{b}}-1)=0.
                          \end{array}
                        \right.\end{equation}
                        Using (\ref{EQ4}),  we can see that 
                        \begin{equation}\label{EQ7}\left\{
                          \begin{array}{ll}
                            a\overline{b}e^{a\overline{b}}+m(e^{a\overline{b}}-1) =0, \\
                            (a\overline{b})^2e^{a\overline{b}}+2m a\overline{b}e^{a\overline{b}}-(m-1)a\overline{b}e^{a\overline{b}}=0.
                          \end{array}
                        \right.\end{equation}
By (\ref{EQ7}),
\begin{equation*}\left\{
                          \begin{array}{ll}
                            a\overline{b}=-(m+1), \\
                     e^{a\overline{b}}=-m.                          \end{array}
                        \right.\end{equation*}
                        Thus, \begin{equation*}\left\{
                          \begin{array}{ll}
                            a\overline{b}=-(m+1), \\
                            e^{a\overline{b}}-a\overline{b}-1=0.                          \end{array}
                        \right.\end{equation*}
This contradiction implies that $ab=0$.
 This proof is complete.
 \end{proof}
 \section{Concluding Remarks}
 In this final section, we consider the possibility of extending the result to other symbol spaces.

\begin{theorem}\label{T1}
 Suppose $f,g\in \mathcal{P}.$ Then $H^*_{\overline{f}}H_{\overline{g}}=0$ on $F^{2,m}(\mathbb{C})$ if and only if at least one of  $f$ and $g$ is a constant function.
  \end{theorem}
We omit the proof of Theorem \ref{T1} here,  since
it is a slight improvement of result contained in  \cite{Yan}. Define $$\mathcal{A}=\left\{\sum_{i=1}^N p_i(z) K_m(z,a_i): p_i\in \mathcal{P} ~\text{and}~ a_j\in \mathbb{C}~ \text{for}~ i=1,\cdots, N\right\}.$$  It is easy to see that
$$e^{\overline{a}z}=(\overline{a}z)^mK_{m}(z,a)+q_m(\overline{a}z).$$
This implies that $\mathcal{A}_1\subseteq \mathcal{A}.$
 Note that $\mathcal{P}$ and $D$ are dense in $F^{2,m}(\mathbb{C})$. It would be an interesting
problem to consider the semi-commutativity of the
Toeplitz operators on $F^{2,m}(\mathbb{C})$.  By what was proved in the previous paragraph, we give a conjecture:
\begin{conjecture}
Suppose $f$ and $g$ be functions in $\mathcal{A},$ then $H^*_{\overline{f}}H_{\overline{g}}=0$ on $F^{2,m}(\mathbb{C})$ if and only if at least one of $f$ and $g$ is a constant function.
\end{conjecture}


\end{document}